\documentclass[12pt, reqno]{amsart}
\usepackage{amsmath, amsthm, amscd, amsfonts, amssymb, graphicx}

\newtheorem{theorem}{Theorem}[section]
\newtheorem{lemma}[theorem]{Lemma}

\newtheorem{corollary}[theorem]{Corollary}
\theoremstyle{definition}

\theoremstyle{remark}
\newtheorem{remark}[theorem]{Remark}
\numberwithin{equation}{section}

\input{mathrsfs.sty}

\begin{document}
\title[A Diaz--Metcalf type inequality for positive linear maps]{A Diaz--Metcalf type inequality for positive linear maps and its applications}
\author[M.S. Moslehian, R. Nakamoto, Y. Seo]{Mohammad Sal Moslehian$^1$, Ritsuo Nakamoto$^2$ and Yuki Seo$^3$}

\address{$^1$ Department of Pure Mathematics, Centre of Excellence in Analysis on
Algebraic Structures (CEAAS), Ferdowsi University of Mashhad, P.O. Box 1159,
Mashhad 91775, Iran.}
\email{moslehian@ferdowsi.um.ac.ir and moslehian@ams.org}

\address{$^2$ Faculty of Engineering, Ibaraki University, Hitachi, Ibaraki 316-0033, Japan.}
\email{r-naka@net1.jway.ne.jp}

\address{$^3$ Faculty of Engineering, Shibaura Institute of Technology, 307 Fukasaku, Minuma-ku, Saitama-city, Saitama 337-8570, Japan.}
\email{yukis@sic.shibaura-it.ac.jp}


\begin{abstract}
We present a Diaz--Metcalf type operator inequality as a reverse Cauchy--Schwarz inequality and then apply it to get the operator versions of P\'{o}lya--Szeg\"{o}'s, Greub--Rheinboldt's, Kantorovich's, Shisha--Mond's, Schweitzer's, Cassels' and Klamkin--McLenaghan's inequalities via a unified approach. We also give some operator Gr\"uss type inequalities and an operator Ozeki--Izumino--Mori--Seo type inequality. Several applications are concluded as well.
\end{abstract}

\maketitle


\section{Introduction}

The Cauchy--Schwarz inequality plays an essential role in mathematical inequalities and its applications. In a semi-inner product space $(\mathscr{H}, \langle \cdot,\cdot\rangle)$ the Cauchy--Schwarz inequality reads as follows
\begin{equation*}
|\langle x,y\rangle|\leq \langle x,x\rangle^{1/2} \langle y,y\rangle^{1/2} \qquad (x,y \in \mathscr{H}).
\end{equation*}
There are interesting generalizations of the Cauchy--Schwarz inequality in various frameworks, e.g. finite sums, integrals, isotone functionals, inner product spaces, $C^*$-algebras and Hilbert $C^*$-modules; see \cite{DRA1, DRA2, DRA3, JOI, MP, NIC, F-H-P-S} and references therein. There are several reverses of the Cauchy--Schwarz  inequality in the literature: Diaz--Metcalf's, P\'{o}lya--Szeg\"{o}'s, Greub--Rheinboldt's, Kantorovich's, Shisha--Mond's, Ozeki--Izumino--Mori--Seo's, Schweitzer's, Cassels' and Klamkin--McLenaghan's inequalities.

Inspired by the work of J.B. Diaz and F.T. Metcalf \cite{DM}, we present several reverse Cauchy--Schwarz type inequalities for positive linear maps.  We give a unified treatment of some reverse inequalities of the classical Cauchy--Schwarz type for positive linear maps.

Throughout the paper $\mathbb{B}(\mathscr{H})$ stands for the algebra of all bounded linear operators acting on a Hilbert space $\mathscr{H}$. We simply denote by $\alpha$ the scalar multiple $\alpha I$ of the identity operator $I\in \mathbb{B}(\mathscr{H})$. For self-adjoint operators $A, B$ the partially ordered relation $B \leq A$ means
that $\langle B\xi,\xi\rangle\leq \langle A\xi,\xi\rangle$ for all $\xi\in\mathscr{H}$. In particular, if $0 \leq
A$, then $A$ is called positive. If $A$ is a positive invertible operator, then we write $0<A$. A linear map $\Phi: {\mathscr A} \to {\mathscr B}$ between $C^*$-algebras is said to be positive if $\Phi(A)$ is positive whenever so is $A$. We say that $\Phi$ is unital if $\Phi$ preserves the identity. The reader is referred to \cite{F-H-P-S} for undefined notations and terminologies.


\section{Operator Diaz--Metcalf type inequality}

We start this section with our main result. Recall that the geometric operator mean $A\ \sharp  \ B$ for positive operators $A, B \in \mathbb{B}(\mathscr{H})$ is defined by
  \[
  A\ \sharp  \ B = A^{\frac{1}{2}}\left( A^{-\frac{1}{2}}BA^{-\frac{1}{2}}\right)^{\frac{1}{2}}A^{\frac{1}{2}}\,
  \]
if $0<A$.
\begin{theorem}
Let $A, B \in \mathbb{B}(\mathscr{H})$ be positive invertible operators and $\Phi: \mathbb{B}(\mathscr{H}) \to \mathbb{B}(\mathscr{K})$ be a positive linear map.

\quad {\rm (i)} If $m^2A \leq B \leq M^2A$ for some positive real numbers $m<M$, then the following inequalities hold:
\begin{itemize}
\item Operator Diaz--Metcalf inequality of first type
\begin{eqnarray*}
Mm\Phi(A)+\Phi(B) \leq (M+m)\Phi(A\sharp B)\,;
\end{eqnarray*}

\item Operator Cassels inequality
\begin{eqnarray*}
\Phi(A)\sharp \Phi(B) \leq \frac{M+m}{2\sqrt{Mm}}\Phi(A\sharp B)\,;
\end{eqnarray*}

\item Operator Klamkin--McLenaghan inequality
\begin{eqnarray*}
\Phi(A\sharp B)^{\frac{-1}{2}}\Phi(B)\Phi(A\sharp B)^{\frac{-1}{2}}- \Phi(A\sharp B)^{\frac{1}{2}}\Phi(A)^{-1}\Phi(A\sharp B)^{\frac{1}{2}}\leq (\sqrt{M}-\sqrt{m})^2\,;
\end{eqnarray*}
\item Operator Kantorovich inequality
\begin{eqnarray*}
\Phi(A)\sharp \Phi(A^{-1}) \leq \frac{M^2+m^2}{2Mm}\,.
\end{eqnarray*}
\end{itemize}

\quad {\rm (ii)} If $m_1^2 \leq A \leq M_1^2$ and $m_2^2\leq B\leq M_2^2$  for some positive real numbers $m_1<M_1$ and $m_2 <M_2$, then the following inequalities hold:
\begin{itemize}
\item Operator Diaz--Metcalf inequality of second type
\begin{eqnarray*}
\frac{M_2m_2}{M_1m_1}\Phi(A)+\Phi(B) \leq \left(\frac{M_2}{m_1}+\frac{m_2}{M_1}\right)\Phi(A\sharp B)\,;
\end{eqnarray*}

\item Operator P\'{o}lya--Szeg\"{o} inequality
\begin{eqnarray*}
\Phi(A)\sharp \Phi(B) \leq \frac{1}{2}\left(\sqrt{\frac{M_1M_2}{m_1m_2}}+\sqrt{\frac{m_1m_2}{M_1M_2}}\right)\Phi(A\sharp B)\,;
\end{eqnarray*}

\item Operator Shisha--Mond inequality
\begin{eqnarray*}
\Phi(A\sharp B)^{\frac{-1}{2}}\Phi(B)\Phi(A\sharp B)^{\frac{-1}{2}}- \Phi(A\sharp B)^{\frac{1}{2}}\Phi(A)^{-1}\Phi(A\sharp B)^{\frac{1}{2}}
\leq \left(\sqrt{\frac{M_2}{m_1}}-\sqrt{\frac{m_2}{M_1}}\right)^2\,;
\end{eqnarray*}
\item Operator Gr\"uss type inequality
\begin{eqnarray*}
\Phi(A)\sharp \Phi(B)- \Phi(A\sharp B)\leq \frac{\sqrt{M_1M_2}\left(\sqrt{M_1M_2}-\sqrt{m_1m_2}\right)^2}{2\sqrt{m_1m_2}}\min\left\{\frac{M_1}{m_1}, \frac{M_2}{m_2}\right\}\,.
\end{eqnarray*}
\end{itemize}
\end{theorem}

\begin{proof}
\quad (i) If $m^2A \leq B \leq M^2A$ for some positive real numbers $m<M$, then $m^2 \leq A^{\frac{-1}{2}}BA^{\frac{-1}{2}}\leq M^2$.

\quad(ii) If $m_1^2 \leq A \leq M_1^2$ and $m_2^2\leq B\leq M_2^2$  for some positive real numbers $m_1<M_1$ and $m_2<M_2$, then
\begin{eqnarray}\label{0}
m^2=\frac{m_2^2}{M_1^2} \leq A^{\frac{-1}{2}}BA^{\frac{-1}{2}} \leq \frac{M_2^2}{m_1^2}=M^2\,.
\end{eqnarray}
In any case we then have
$$\left(M-\left(A^{\frac{-1}{2}}BA^{\frac{-1}{2}}\right)^{1/2}\right)\left(\left(A^{\frac{-1}{2}}BA^{\frac{-1}{2}}\right)^{1/2}-m\right)\geq 0\,,$$
whence
$$Mm+A^{\frac{-1}{2}}BA^{\frac{-1}{2}} \leq (M+m)\left(A^{\frac{-1}{2}}BA^{\frac{-1}{2}}\right)^{\frac{1}{2}}\,.$$
Hence
\begin{eqnarray}\label{2.1}
MmA+B \leq (M+m)A^{1/2}\left(A^{\frac{-1}{2}}BA^{\frac{-1}{2}}\right)^{\frac{1}{2}}A^{1/2}=(M+m)A\sharp B\,.
\end{eqnarray}
Since $\Phi$ is a positive linear map, \eqref{2.1} yields the \textit{operator Diaz--Metcalf inequality of first type} as follows:
\begin{eqnarray}\label{2.2}
Mm\Phi(A)+\Phi(B) \leq (M+m)\Phi(A\sharp B)\,.
\end{eqnarray}
In the case when (ii) holds we get the following, which is called the \textit{operator Diaz--Metcalf inequality of second type}:
\begin{eqnarray*}
\frac{M_2m_2}{M_1m_1}\Phi(A)+\Phi(B) \leq \left(\frac{M_2}{m_1}+\frac{m_2}{M_1}\right)\Phi(A\sharp B)\,.
\end{eqnarray*}
Following the strategy of \cite{NIE}, we apply the operator geometric--arithmetic inequality to $Mm\Phi(A)$ and $\Phi(B)$ to get:
\begin{eqnarray}\label{2.4}
\sqrt{Mm}(\Phi(A)\sharp \Phi(B))&=&(Mm\Phi(A)) \sharp \Phi(B)\leq \frac{1}{2}\left(Mm\Phi(A)+ \Phi(B)\right).
\end{eqnarray}
It follows from \eqref{2.2} and \eqref{2.4} that
\begin{eqnarray*}
\Phi(A)\sharp \Phi(B) \leq \frac{M+m}{2\sqrt{Mm}}\Phi(A\sharp B)\,,
\end{eqnarray*}
which is said to be the \textit{operator Cassels inequality} under the assumption (i); see also \cite{LEE}. Under the case (ii) we can represent it as the following inequality being called the \textit{operator P\'olya--Szeg\"o inequality} or the \textit{operator Greub--Rheinboldt inequality}:
\begin{eqnarray}\label{2.6}
\Phi(A)\sharp \Phi(B) \leq \frac{1}{2}\left(\sqrt{\frac{M_1M_2}{m_1m_2}}+\sqrt{\frac{m_1m_2}{M_1M_2}}\right)\Phi(A\sharp B)\,.
\end{eqnarray}
It follows from \eqref{2.6} that
\begin{eqnarray}\label{001}
\Phi(A)\sharp \Phi(B)- \Phi(A\sharp B)&\leq& \left( \frac{1}{2}\left(\sqrt{\frac{M_1M_2}{m_1m_2}}+\sqrt{\frac{m_1m_2}{M_1M_2}}\right)-1\right)\Phi(A\sharp B)\nonumber\\
&=& \frac{\left(\sqrt{M_1M_2}-\sqrt{m_1m_2}\right)^2}{2\sqrt{m_1m_2}\sqrt{M_1M_2}}\Phi(A\sharp B)\,.
\end{eqnarray}
It follows from \eqref{0} that
\begin{eqnarray*}
\frac{m_2}{M_1}A \leq A^{1/2}\left(A^{\frac{-1}{2}}BA^{\frac{-1}{2}}\right)^{1/2}A^{1/2} \leq \frac{M_2}{m_1}A\,,
\end{eqnarray*}
so
\begin{eqnarray}\label{002}
\frac{m_1^2m_2}{M_1} \leq A\sharp B \leq \frac{M_1^2M_2}{m_1}\,.
\end{eqnarray}
Now, \eqref{001} and \eqref{002} yield that
\begin{eqnarray*}
\Phi(A)\sharp \Phi(B)- \Phi(A\sharp B)\leq \frac{\left(\sqrt{M_1M_2}-\sqrt{m_1m_2}\right)^2}{2\sqrt{m_1m_2}\sqrt{M_1M_2}} \frac{M_1^2M_2}{m_1}\,.
\end{eqnarray*}
An easy symmetric argument then follows that
\begin{eqnarray*}
\Phi(A)\sharp \Phi(B)- \Phi(A\sharp B)\leq \frac{\sqrt{M_1M_2}\left(\sqrt{M_1M_2}-\sqrt{m_1m_2}\right)^2}{2\sqrt{m_1m_2}}\min\left\{\frac{M_1}{m_1}, \frac{M_2}{m_2}\right\}\,,
\end{eqnarray*}
presenting a Gr\"uss type inequality.

If $A$ is invertible and $\Phi$ is unital and $m_1^2=m^2\leq A \leq M^2=M_1^2$, then by putting $m_2^2=1/M^2\leq B = A^{-1}\leq 1/m^2=M_2^2$ in \eqref{2.6} we get the following \textit{operator Kantorovich inequality}:
\begin{eqnarray}\label{2.7}
\Phi(A)\sharp \Phi(A^{-1}) \leq \frac{M^2+m^2}{2Mm}\,.
\end{eqnarray}
It follows from \eqref{2.2} that
\begin{eqnarray}\label{2.8}
&&\Phi(A\sharp B)^{\frac{-1}{2}}\Phi(B)\Phi(A\sharp B)^{\frac{-1}{2}}- \Phi(A\sharp B)^{\frac{1}{2}}\Phi(A)^{-1}\Phi(A\sharp B)^{\frac{1}{2}}\nonumber\\
&\leq& M+m-Mm\Phi(A\sharp B)^{\frac{-1}{2}}\Phi(A)\Phi(A\sharp B)^{\frac{-1}{2}}- \Phi(A\sharp B)^{\frac{1}{2}}\Phi(A)^{-1}\Phi(A\sharp B)^{\frac{1}{2}}\nonumber\\
&\leq& M+m-2\sqrt{Mm}-\Big(\sqrt{Mm}\left(\Phi(A\sharp B)^{\frac{-1}{2}}\Phi(A)\Phi(A\sharp B)^{\frac{-1}{2}}\right)^{1/2}\nonumber\\
&&-\left(\Phi(A\sharp B)^{\frac{1}{2}}\Phi(A)^{-1}\Phi(A\sharp B)^{\frac{1}{2}}\right)^{1/2}\Big)^2\nonumber\\
&\leq& (\sqrt{M}-\sqrt{m})^2\,,
\end{eqnarray}
that is, an \textit{operator Klakmin--Mclenaghan inequality} when (i) holds. Under (ii), we get the following \textit{operator Shisha--Szeg\"o inequality} from \eqref{2.8}:
\begin{eqnarray*}
\Phi(A\sharp B)^{\frac{-1}{2}}\Phi(B)\Phi(A\sharp B)^{\frac{-1}{2}}- \Phi(A\sharp B)^{\frac{1}{2}}\Phi(A)^{-1}\Phi(A\sharp B)^{\frac{1}{2}}
\leq \left(\sqrt{\frac{M_2}{m_1}}-\sqrt{\frac{m_2}{M_1}}\right)^2\,.
\end{eqnarray*}
\end{proof}

\section{Applications}
If $(a_1, \ldots, a_n)$ and $(b_1, \ldots, b_n)$ are  $n$-tuples of real
numbers with $0< m_1 \leq a_i \leq M_1\,\,\,(1\leq i\leq n), 0< m_2 \leq b_i \leq M_2\,\,(1\leq i\leq n)$, we can consider the positive linear map $\Phi(T)=\langle Tx,x\rangle$ on $\mathbb{B}(\mathbb{C}^n)=M_n(\mathbb{C})$ and let $A={\rm diag}(a_1^2, \ldots, a_n^2)$, $B={\rm diag}(b_1^2, \ldots, b_n^2)$ and $x=(1, \ldots, 1)^t$ in the operator inequalities above to get the following classical inequalities:
\begin{itemize}
\item Diaz--Metcalf inequality \cite{DM}
\begin{eqnarray*}
\sum_{k=1}^{n}b_{k}^{2}+\frac{m_{2}M_{2}}{m_{1}M_{1}}\sum_{k=1}^{n}a_{k}^{2}
\leq \left( \frac{M_{2}}{m_{1}}+\frac{m_{2}}{M_{1}}\right)
\sum_{k=1}^{n}a_{k}b_{k}\,.
\end{eqnarray*}

\item P\'{o}lya--Szeg\"{o} inequality \cite{PS}
\begin{eqnarray*}
\frac{\sum_{k=1}^{n}a_{k}^{2}\sum_{k=1}^{n}b_{k}^{2}}{\left(
\sum_{k=1}^{n}a_{k}b_{k}\right) ^{2}}\leq \frac{1}{4}\left( \sqrt{\frac{
M_{1}M_{2}}{m_{1}m_{2}}}+\sqrt{\frac{m_{1}m_{2}}{M_{1}M_{2}}}\right) ^{2}\,;
\end{eqnarray*}

\item Shisha--Mond inequality \cite{SM}
\begin{eqnarray*}
\frac{\sum_{k=1}^{n}a_{k}^{2}}{\sum_{k=1}^{n}a_{k}b_{k}}-\frac{
\sum_{k=1}^{n}a_{k}b_{k}}{\sum_{k=1}^{n}b_{k}^{2}}\leq \left( \sqrt{\frac{
M_{1}}{m_{2}}}-\sqrt{\frac{m_{1}}{M_{2}}}\right) ^{2}\,;
\end{eqnarray*}

\item A Gr\"uss type inequality
\begin{eqnarray*}
\left(\sum_{k=1}^{n}a_{k}^{2}\right)^{1/2}\left(\sum_{k=1}^{n}b_{k}^{2}\right)^{1/2}-\sum_{k=1}^{n}a_{k}b_{k}\quad\qquad\qquad\qquad\qquad\qquad\qquad\qquad\qquad\qquad\\
\qquad\qquad\qquad\leq \frac{\sqrt{M_1M_2}\left(\sqrt{M_1M_2}-\sqrt{m_1m_2}\right)^2}{2\sqrt{m_1m_2}}\min\left\{\frac{M_1}{m_1}, \frac{M_2}{m_2}\right\}\,.
\end{eqnarray*}
\end{itemize}
Using the same argument with a positive $n$-tuple $(a_1, \ldots, a_n)$ of real
numbers with $0< m \leq a_i \leq M\,\,(1 \leq i \leq n)$, $x=\frac{1}{\sqrt{n}}(1, \ldots, 1)^t$, we get from Kantorovich inequality that
\begin{itemize}
\item Schweitzer inequality \cite{BUL}
\begin{eqnarray*}
\left(\frac{1}{n}\sum_{i=1}^na_i^2\right)\left(\frac{1}{n}\sum_{i=1}^n a_i^{-2}\right) \leq \frac{(M^2+m^2)^2}{4M^2m^2}\,.
\end{eqnarray*}
\end{itemize}

If $(a_1, \ldots, a_n)$ and $(b_1, \ldots, b_n)$ are $n$-tuples of real
numbers with $0< m \leq a_i/b_i \leq M\,\,\,(1\leq i\leq n)$, we can consider the positive linear map $\Phi(T)=\langle Tx,x\rangle$ on $\mathbb{B}(\mathbb{C}^n)=M_n(\mathbb{C})$ and let $A={\rm diag}(a_1^2, \ldots, a_n^2)$, $B={\rm diag}(b_1^2, \ldots, b_n^2)$ and $x=(\sqrt{w_1}, \ldots, \sqrt{w_n})^t$ based on the weight $\mathbf{\bar{w}}=\left( w_{1},\dots ,w_{n}\right)$, in the operator inequalities above to get the following classical inequalities:

\begin{itemize}
\item Cassels inequality \cite{WAT}
\begin{eqnarray*}
\frac{\sum_{k=1}^{n}w_{k}a_{k}^{2}\sum_{k=1}^{n}w_{k}b_{k}^{2}}{\left(
\sum_{k=1}^{n}w_{k}a_{k}b_{k}\right) ^{2}}\leq \frac{\left( M+m\right) ^{2}}{
4mM}\,;
\end{eqnarray*}

\item Klamkin--McLenaghan inequality \cite{K-M}

\begin{eqnarray*}
\sum_{k=1}^{n}w_{k}a_{k}^{2}\sum_{k=1}^{n}w_{k}b_{k}^{2}-\left(
\sum_{k=1}^{n}w_{k}a_{k}b_{k}\right) ^{2}\leq \left( \sqrt{M}-\sqrt{m}
\right) ^{2}\sum_{k=1}^{n}w_{k}a_{k}b_{k}\sum_{k=1}^{n}w_{k}a_{k}^{2}.
\end{eqnarray*}
\end{itemize}

Using the same argument, we obtain a weighted form of the P\'{o}lya--Szeg\"{o} inequality as follows:
\begin{itemize}
\item Grueb--Rheinboldt inequality \cite{G-R}
\begin{eqnarray*}
\frac{\sum_{k=1}^{n}w_{k}a_{k}^{2}\sum_{k=1}^{n}w_{k}b_{k}^{2}}{\left(
\sum_{k=1}^{n}w_{k}a_{k}b_{k}\right) ^{2}}\leq \frac{\left(
M_{1}M_{2}+m_{1}m_{2}\right) ^{2}}{4m_{1}m_{2}M_{1}M_{2}}.
\end{eqnarray*}
\end{itemize}


One can assert the integral versions of discrete results above by considering $L^2(X, \mu)$, where $(X, \mu)$ is a probability space, as a Hilbert space via $\langle h_1,h_2\rangle=\int_X h_1\overline{h_2} d\mu$, multiplication operators $A, B \in \mathbb{B}(L^2(X, \mu)))$ defined by $A(h)=f^2h$ and $B(h)=g^2h$ for bounded $f, g \in L^2(X, \mu)$ and a positive linear map $\Phi$ by $\Phi(T)=\int_XT(1)d\mu$ on $\mathbb{B}(L^2(X, \mu)))$. For instance, let us state integral versions of the Cassels and Klamkin--McLenaghan inequalities. These two inequalities are obtained, first for bounded positive functions $f, g \in L^2(X, \mu)$ and next for general positive functions $f, g \in L^2(X, \mu)$ as the limits of sequences of bounded positive functions.
\begin{corollary}
Let $(X, \mu)$ be a probability space and $f, g \in
L^2(X,\mu)$ with $ 0 \leq mg \leq f \leq Mg$ for some scalars $0<m<M$.
Then
\begin{eqnarray*}
\int_X f^2 d\mu \int_X g^2 d\mu \leq  \frac{(M+m)^2}{4Mm} \left(\int_X fg d\mu\right)^2
\end{eqnarray*}
and
\begin{eqnarray*}
\int_X f^2 d\mu \int_X g^2 d\mu -\left(\int_X fg d\mu\right)^2 \leq  \left( \sqrt{M}-\sqrt{m}\right)^2 \int_X fg d\mu \int_X f^2 d\mu\,.
\end{eqnarray*}
\end{corollary}


Considering the positive linear functional $\Phi(R)=\sum_{i=1}^n\langle R\xi_i,\xi_i\rangle$ on $\mathbb{B}(\mathscr{H})$, where $\xi_1, \ldots, \xi_n \in \mathscr{H}$, we get the following versions of the Diaz--Metcalf and P\'{o}lya--Szeg\"{o} inequalities in a Hilbert space.


\begin{corollary}
Let $\mathscr{H}$ be a Hilbert space, let $\xi_1, \ldots, \xi_n \in \mathscr{H}$ and let $T, S \in \mathbb{B}(\mathscr{H})$ be positive operators satisfying $0<m_1\leq T\leq M_1$ and $0<m_2\leq S \leq M_2$. Then
\begin{eqnarray*}
\frac{M_2m_2}{M_1m_1}\sum_{i=1}^n\|T\xi_i\|^2+\sum_{i=1}^n\|S\xi_i\|^2 \leq \left(\frac{M_2}{m_1}+\frac{m_2}{M_1}\right)\sum_{i=1}^n\| (T^2\sharp S^2)^{1/2}\xi_i\|^2
\end{eqnarray*}
and
\begin{eqnarray*}
&&\left(\sum_{i=1}^n\|T\xi_i\|^2\right)^{1/2}\left(\sum_{i=1}^n\|S\xi_i\|^2\right)^{1/2}\\ &&\qquad\qquad\qquad\qquad\qquad\qquad\leq \frac{1}{2}\left(\sqrt{\frac{M_1M_2}{m_1m_2}}+\sqrt{\frac{m_1m_2}{M_1M_2}}\right)\sum_{i=1}^n\| (T^2\sharp S^2)^{1/2}\xi_i\|^2\,.
\end{eqnarray*}
\end{corollary}

\par
\bigskip


\section{A Gr\"uss type inequality}


In this section we obtain another Gr\"uss type inequality, see also \cite{m-r}. Let $\mathscr{A}$ be a $C^{*}$-algebra and let $\mathscr{B}$ be a $C^{*}$-subalgebra of $\mathscr{A}$. Following \cite{A-B-M}, a positive linear map $\Phi : \mathscr{A} \to \mathscr{B}$ is called a left multiplier if $\Phi(XY)=\Phi(X)Y$ for every $X\in \mathscr{A}$, $Y\in \mathscr{B}$.  \par

The following lemma is interesting on its own right.

\begin{lemma}\label{lem1}
Let $\Phi$ be a unital positive linear map on $\mathscr{A}$, $A \in \mathscr{A}$ and
$M, m$ be complex numbers such that
\begin{eqnarray}\label{1.0}
{\rm Re}\left((M-A)^*(A-m)\right)\geq 0\,.
\end{eqnarray}
Then
\begin{eqnarray}\label{2.0}
\Phi(|A|^2) - \left|\Phi(A)\right|^2  \leq \frac{1}{4}|M-m|^2\,.
\end{eqnarray}
\end{lemma}

\begin{proof}
 For any complex number $c\in {\Bbb C}$, we have
 \begin{equation} \label{eq:Gr-1}
 \Phi(|A|^2)-|\Phi(A)|^2 = \Phi(|A-c|^2)-|\Phi(A-c)|^2.
 \end{equation}
 Since for any $T\in \mathscr{A}$ the operator equality
\[
\frac{1}{4}|M-m|^2-\left|T-\frac{M+m}{2}\right|^2=\mbox{Re}\left((M-T)(T-m)^*\right)
\]
holds, the condition \eqref{1.0} implies that
\begin{equation} \label{eq:Gr-2}
\Phi\left(\left|A-\frac{M+m}{2}\right|^2\right) \leq \frac{1}{4}|M-m|^2\,.
\end{equation}
Therefore, it follows from \eqref{eq:Gr-1} and \eqref{eq:Gr-2} that
\begin{align*}
\Phi(|A|^2)-|\Phi(A)|^2 & \leq \Phi(|A-\frac{M+m}{2}|^2) \\
& \leq \frac{1}{4}|M-m|^2.
\end{align*}
\end{proof}

\begin{remark}
If (i) $\Phi$ is a unital positive linear map and $A$ is a normal operator or (ii) $\Phi$ is a $2$-positive linear map and $A$ is an arbitrary operator, then it follows from \cite{CHO} that
\begin{eqnarray}\label{med}
0 \leq \Phi(|A|^2) - \left|\Phi(A)\right|^2 \,.
\end{eqnarray}
Condition \eqref{med} is stronger than positivity and weaker than $2$-positivity; see \cite{E-L}. Another class of positive linear maps satisfying \eqref{med} are left multipliers, cf. \cite[Corollary 2.4]{A-B-M}.
\end{remark}

\begin{lemma}\label{lem2}
Let a positive linear map $\Phi: \mathscr{A} \to \mathscr{B}$ be a unital left multiplier. Then
\begin{eqnarray}\label{3.0}
\left|\Phi(A^*B)-\Phi(A)^*\Phi(B)\right|^2\leq \left\|\Phi(|A|^2)-|\Phi(A)|^2\right\|\,\left( \Phi(|B|^2)-|\Phi(B)|^2\right)
\end{eqnarray}
\end{lemma}
\begin{proof}
If we put $[X,Y]:=\Phi(X^{*}Y)-\Phi(X)^*\Phi(Y)$, then $\mathscr{A}$ is a right pre-inner product $C^{*}$-module over $\mathscr{B}$, since $\Phi(X^{*}Y)$ is a right pre-inner product $\mathscr{B}$-module, see \cite[Corollary 2.4]{A-B-M}. It follows from the Cauchy--Schwarz inequality in pre-inner product $C^{*}$-modules (see \cite[Proposition 1.1]{LAN}) that
\begin{align*}
\left|\Phi(A^*B)-\Phi(A)^*\Phi(B)\right|^2& = [B,A][A,B]\\
& \leq \| [A,A]\| [B,B] \\
& = \| \Phi(A^{*}A)-\Phi(A)^{*}\Phi(A) \| \left( \Phi(B^{*}B)-\Phi(B)^{*}\Phi(B) \right)
\end{align*}
and hence \eqref{3.0} holds.
\end{proof}
\begin{theorem}
Let a positive linear map $\Phi: \mathscr{A} \to \mathscr{B}$ be a unital left multiplier. If $M_1, m_1, M_2, m_2 \in {\Bbb C}$ and $A,B\in \mathscr{A}$ satisfy the following conditions:
\[
{\rm Re}(M_1-A)^*(A-m_1)\geq 0 \quad \mbox{and} \quad {\rm Re}(M_2-B)^*(B-m_2)\geq 0,
\]
then
\begin{eqnarray*}
\left| \Phi(A^*B)-\Phi(A)^* \Phi(B)\right| \leq \frac{1}{4} |M_1-m_1| \, |M_2-m_2|.
\end{eqnarray*}
\end{theorem}

\begin{proof}
By L\"{o}wner--Heinz theorem, we have
\begin{eqnarray*}
&&\left| \Phi(A^*B)-\Phi(A)^*\Phi(B) \right| \\
&\leq& \left\|\Phi(|A|^2)-|\Phi(A)|^2\right\|^{\frac{1}{2}} \, \left( \Phi(|B|^2)-|\Phi(B)|^2 \right)^{\frac{1}{2}} \qquad\quad\qquad\quad (\mbox{by~Lemma}~\ref{lem2})  \\
&\leq& \frac{1}{4} |M_1-m_1|\, |M_2-m_2| \qquad\qquad\qquad\qquad\qquad\qquad\qquad\quad (\mbox{by~Lemma}~\ref{lem1})\,.
\end{eqnarray*}
\end{proof}


\section{Ozeki--Izumino--Mori--Seo type inequality}


 Let $a=(a_1, \ldots , a_n)$ and $b=(b_1, \ldots , b_n)$ be $n$-tuples of real numbers satisfying
 \[
 0\leq m_1\leq a_i\leq M_1 \quad \mbox{and} \quad 0\leq m_2\leq b_i\leq M_2 \quad (i=1, \ldots , n).
 \]

 Then Ozeki--Izumino--Mori--Seo inequality \cite{IMS, Ozeki} asserts that
  \begin{equation} \label{eq:OI}
  \sum_{i=1}^n a_i^2 \sum_{i=1}^n b_i^2 - \left( \sum_{i=1}^n a_ib_i \right)^2 \leq \frac{n^2}{3} \left( M_1M_2-m_1m_2\right)^2.
  \end{equation}

In \cite{IMS} they also showed the following operator version of \eqref{eq:OI}: If $A$ and $B$ are positive operators in $\mathbb{B}(\mathscr{H})$ such that $0<m_1\leq A\leq M_1$ and $0<m_2\leq B\leq M_2$ for some scalars $m_1\leq M_1$ and $m_2\leq M_2$, then
\begin{equation} \label{eq:OOI}
(A^2x,x)(B^2x,x)-(A^2\ \sharp \ B^2 x,x)^2 \leq \frac{1}{4\gamma^2} \left( M_1M_2-m_1m_2\right)^2
\end{equation}
for every unit vector $x\in H$, where $\gamma = \max \{ \frac{m_1}{M_1}, \frac{m_2}{M_2} \}$.\par
Based on the Kantorovich inequality for the difference, we present an extension of Ozeki--Izumino--Mori--Seo inequality \eqref{eq:OOI} as follows.
\begin{theorem}
Suppose that $\Phi: \mathbb{B}(\mathscr{H}) \to \mathbb{B}(\mathscr{K})$ is a positive linear map such that $\Phi(I)$ is invertible and $\Phi(I)\leq I$. Assume that $A, B \in \mathbb{B}(\mathscr{H})$ are positive invertible operators such that $0<m_1\leq A\leq M_1$ and $0<m_2\leq B\leq M_2$ for some scalars $m_1\leq M_1$ and $m_2\leq M_2$. Then
\begin{equation} \label{eq:phi-1}
\Phi(B^2)^{\frac{1}{2}} \Phi(A^2)\Phi(B^2)^{\frac{1}{2}} - | \Phi(B^2)^{-\frac{1}{2}} \Phi(A^2\sharp B^2) \Phi(B^2)^{\frac{1}{2}} | ^2 \leq \frac{(M_1M_2-m_1m_2)^2}{4} \times \frac{M_2^2}{m_2^2}
\end{equation}
and
\begin{equation} \label{eq:phi-2}
\Phi(A^2)^{\frac{1}{2}} \Phi(B^2)\Phi(A^2)^{\frac{1}{2}} - | \Phi(A^2)^{-\frac{1}{2}} \Phi(A^2\sharp B^2) \Phi(A^2)^{\frac{1}{2}} |^2 \leq \frac{(M_1M_2-m_1m_2)^2}{4} \times \frac{M_1^2}{m_1^2}.
\end{equation}
\end{theorem}

\begin{proof}
Define a normalized positive linear map $\Psi$ by
\begin{equation*}
\Psi(X):=\Phi(A)^{-\frac{1}{2}}\Phi(A^{\frac{1}{2}}XA^{\frac{1}{2}})\Phi(A)^{-\frac{1}{2}}.
\end{equation*}
By using the Kantorovich inequality for the difference, it follows that
\begin{equation} \label{eq:2}
\Psi(X^2)-\Psi(X)^2 \leq \frac{(M-m)^2}{4}
\end{equation}
for  all $0<m\leq X\leq M$ with some scalars $m\leq M$. As a matter of fact, we have
\begin{align*}
\Psi(X^2)-\Psi(X)^2 & \leq \Psi((M+m)X-Mm)-\Psi(X)^2 \\
& = -\left( \Psi(X)-\frac{M+m}{2} \right)^2 + \frac{(M-m)^2}{4} \\
& \leq \frac{(M-m)^2}{4}.
\end{align*}

If we put $X=(A^{-\frac{1}{2}}BA^{-\frac{1}{2}})^{\frac{1}{2}}$, then due to
\[
0< (m=) \sqrt{\frac{m_2}{M_1}} \leq X \leq \sqrt{\frac{M_2}{m_1}} (=M)
\]
we deduce from \eqref{eq:2} that
\[
\Phi(A)^{-\frac{1}{2}}\Phi(B)\Phi(A)^{-\frac{1}{2}}-\left( \Phi(A)^{-\frac{1}{2}}\Phi(A\sharp B)\Phi(A)^{-\frac{1}{2}}\right)^2 \leq \frac{(\sqrt{M_1M_2}-\sqrt{m_1m_2})^2}{4M_1m_1}.
\]
Pre- and post-multiplying both sides by $\Phi(A)$, we obtain
\begin{align*}
\Phi(A)^{\frac{1}{2}}\Phi(B)\Phi(A)^{\frac{1}{2}}-| \Phi(A)^{-\frac{1}{2}}\Phi(A\sharp B)\Phi(A)^{\frac{1}{2}} |^2 & \leq \frac{(\sqrt{M_1M_2}-\sqrt{m_1m_2})^2}{4M_1m_1} \Phi(A)^2 \\
& \leq \frac{(\sqrt{M_1M_2}-\sqrt{m_1m_2})^2}{4} \times \frac{M_1}{m_1} ,
\end{align*}
since $0 \leq \Phi(A)^2 \leq M_1^2$. Replacing $A$ and $B$ by $A^2$ and $B^2$ respectively, we have the desired inequality \eqref{eq:phi-2}.  Similarly, one can obtain \eqref{eq:phi-1}.
  \end{proof}

\begin{remark}
If $\Phi$ is a vector state in \eqref{eq:phi-1} and \eqref{eq:phi-2}, then we get Ozeki--Izumino--Mori--Seo inequality \eqref{eq:OOI}.
\end{remark}
  \par

\end{document}